\documentclass[journal,twoside,web]{ieeecolor}
\usepackage{generic}
\usepackage{cite}
\usepackage{amsmath,amssymb,amsfonts}
\usepackage{graphicx}
\usepackage{textcomp}
\usepackage{amssymb}
\usepackage{eqnarray} 

\usepackage{amsmath}
\usepackage{mathalfa} 
\usepackage{pdfpages}		
\usepackage{dsfont}   		
\usepackage{algorithm}
\usepackage{algpseudocode}
\usepackage{bbm}
                                               
\usepackage{graphicx}
\usepackage[utf8]{inputenc}
\usepackage{psfrag}
\usepackage{amsmath}
\usepackage{amssymb}
\usepackage{mathtools}

\usepackage{psfrag}
\usepackage{auto-pst-pdf}
\usepackage{url}

\usepackage{cite}

\newcommand{\sD}{\mathcal{D}}
\newcommand{\sS}{\mathcal{S}}

\newcommand{\sT}{\mathcal{T}}

\newcommand{\R}{\mathbb{R}}
\newcommand{\E}{\mathbb{E}}
\newcommand{\Z}{\mathbb{Z}}

\newcommand{\bv}[1]{\mathbf{#1}}
\newcommand{\policy}[2]{\pi_{#1}^{#2}}

\newcommand{\KL}{\text{KL}}

\newcommand{\DKL}{\mathcal{D}_{\KL}}

\newtheorem{Remark}{Remark}
\newtheorem{Theorem}{Theorem}
\newtheorem{Definition}{Definition}
\newtheorem{Lemma}{Lemma}
\newtheorem{Corollary}{Corollary}
\newtheorem{Assumption}{Assumption}

\newtheorem{Problem}{Problem}

\def\BibTeX{{\rm B\kern-.05em{\sc i\kern-.025em b}\kern-.08em
    T\kern-.1667em\lower.7ex\hbox{E}\kern-.125emX}}
\begin{document}
\title{On the crowdsourcing of behaviors for autonomous agents}
\author{Giovanni Russo, \IEEEmembership{Senior Member, IEEE}
\thanks{G. Russo is with the Dept. of Information and Electrical Eng. \& Applied Math. at the University of Salerno, Italy (giovarusso@unisa.it)}}
 
 \pagestyle{empty} 
 
\maketitle

\thispagestyle{empty}

\begin{abstract}
This paper is concerned with the  problem of designing, from data, agents that are able to craft their behavior from a set of {contributors} in order to fulfill some agent-specific task. This is not necessarily known to the contributors. After formalizing this crowdsourcing  process as a control problem, we present a result to synthesize behaviors from the information made available by the contributors. The result is turned into an algorithm and the theoretical findings are complemented via an example.
\end{abstract}


\section{INTRODUCTION}

\IEEEPARstart{W}{hen} we are given a {\em new} task, we can {\em ask others} for more information, crowdsourcing knowledge to achieve our goal. Systems based on crowdsourcing are gaining popularity in several domains: sharing economies \cite{Sharing_book} and blockchain \cite{8496756} are two examples that involve using behaviors, and goods, from others to achieve an individual task. Indeed, agents within a crowdsourcing platform can leverage it to solve tasks with a lower effort than if they were working alone \cite{10.1145/3226028}. 

In this work, we investigate whether decision-making { mechanisms} can be designed to enable autonomous agents to take advantage of the opportunities offered by crowdsourcing. We consider the problem of designing agents that are able to craft their behavior from data made available by others (contributors) so as to fulfill some individual task. We term this problem as the crowdsourcing problem and tackle the question of whether: (i) agents can fulfill a task by using crowdsourced knowledge rather than finding a solution from scratch; (ii) the crowdsourced behavior can outperform the contributors.

\subsubsection*{Related work} federated learning, see e.g. \cite{8664630}, has been proposed to enable a network of agents with a common goal to harmonize and share data models from distributed datasets. In \cite{10.1145/3328526.3329589},  a two-sided data marketplace is formalized and management algorithms are introduced, while techniques allowing robots with the same task to re-use models have been introduced in \cite{5876227}. We also recall works on multi-agent reinforcement learning \cite{Busoniu2010}, where agents collaborate in order to learn a common task, and on preference elicitation  within a probabilistic framework of decision-making \cite{KARNY2019239}, where an individual makes optimal decisions based on his/her own preferences. As we shall see, crowdsourcing can be formalized as a data-driven control problem \cite{8933093,8960476,HOU20133} and recent results include \cite{TANASKOVIC20171}, where a direct data-driven design approach is introduced for stabilizable Lipschitz systems,  \cite{8453019}, where the design of controllers for linear systems is considered, \cite{8933093} that proposes the use of data-dependent linear matrix inequalities and \cite{8039204,8795639} that introduce MPC-inspired data-driven algorithms. 

\subsubsection*{Contributions}  we tackle the crowdsourcing problem for tasks that  require tracking a target behavior while  optimizing an agent-specific reward function. Specifically: (i) we formalize crowdsourcing as a control problem that involves designing the shape of certain probability density/mass functions. {To the best of our knowledge, this is the first paper that proposes a formulation of crowdsourcing as an optimal control problem}; (ii) after characterizing the problem, we present a result to synthesize the agent's behavior from the information made available by the contributors. {The result relies on first approximating (via two lemmas that are also a contribution of this paper) the original problem with a convex problem. Then,  a number of transformations are developed, which split the problem into sub-problems that can be solved recursively. We also turn our main results into an algorithmic procedure;} (iii) finally, we relate the performance obtained by the agent to the performance that would have been obtained by the contributors if they were performing the task, showing that the crowdsourced behavior can outperform the  contributors. The  findings are complemented via an example.

\section{Mathematical Preliminaries}

Sets are in {\em calligraphic} and vectors are in {\bf bold}. Consider the measurable space $(\mathcal{X},\mathcal{F}_x)$, with $\mathcal{X}\subseteq\R^{d}$ ($\mathcal{X}\subseteq\Z^{d}$) and $\mathcal{F}_x$ being a $\sigma$-algebra on $\mathcal{X}$. A random variable on $(\mathcal{X},\mathcal{F}_x)$ is denoted by $\mathbf{X}$ and its realization is $\mathbf{x}$. We denote the \textit{probability density function} (\textit{probability mass function}) or simply \textit{pdf} (\textit{mdf}) of a continuous (discrete) $\mathbf{X}$  by $p(\mathbf{x})$ and we let $\sD$ be the convex subset of density (mass) functions. When we take the integrals (sums) involving pdfs (mdfs) we always assume that the integral (sum) exists. The  expectation of a function $\mathbf{h}(\cdot)$ of $\mathbf{X}$ is $\E_{{p}}[\mathbf{h}(\mathbf{X})]:=\int\mathbf{h}(\mathbf{x})p(\mathbf{x})d\mathbf{x}$, where the integral is over the support of $p(\mathbf{x})$. The \textit{joint} pdf (mdf) of two random vectors, $\mathbf{X}_1$ and $\mathbf{X}_2$, is denoted by  $p(\mathbf{x}_1,\mathbf{x}_2)$ and the \textit{conditional} probability density (mass) function or {\em cpdf} ({\em cpmf}) of $\mathbf{X}_1$ with respect to $\mathbf{X}_2$ is $p\left( \mathbf{x}_1| \mathbf{x}_2 \right)$. Countable sets are denoted by $\lbrace w_k \rbrace_{k_1:k_n}$, where: (i) $w_k$ is the generic set element; (ii) $k_1$, $k_n$ are the indices of the first and last element; (iii) $k_1:k_n$ is the set of consecutive integers between (including) $k_1$ and $k_n$. Pdfs (or pmfs) of the form $p(\bv{x}_0,\ldots,\bv{x}_N)$ are compactly written as $p_{0:N}$ (note that $p_{k:k} := p_k(\bv{x}_k)$). We use the shorthand notation $p_{k|k-1}$ to denote $p_k(\bv{x}_k|\bf{x}_{k-1})$. 
\subsubsection{Useful tools}
consider two pdfs,  $p^{(1)}:=p^{(1)}(\mathbf{z})$ and $p^{(2)}:=p^{(2)}(\mathbf{z})$, with $p^{(1)}$ being absolutely continuous with respect to $p^{(2)}$. The Kullback-Leibler (KL \cite{KL_51}) divergence of $p^{(1)}$ with respect to $p^{(2)}$ is  $\mathcal{D}_{\KL}\left(p^{(1)} || p^{(2)} \right):= \int p^{(1)} \; \ln\left( {p^{(1)}}/{p^{(2)}}\right)\,d\mathbf{z}$.
In case of discrete variables, the KL-divergence between two pmfs is   $\mathcal{D}_{\KL}\left(p^{(1)} || p^{(2)} \right):= \sum p^{(1)} \; \ln\left( {p^{(1)}}/{p^{(2)}}\right)$. For both pdfs and pmfs, $\mathcal{D}_{\KL}\left(p^{(1)} || p^{(2)} \right)$ is a measure of the proximity of the pair $p^{(1)}$ and $p^{(2)}$. Consider two joint pdfs (mdfs), say $p^{(1)}(\bv{x}_0,\ldots,\bv{x}_N)$ and $p^{(2)}(\bv{x}_0,\ldots,\bv{x}_N)$, such that $p^{(1)}(\bv{x}_0,\ldots,\bv{x}_N)  = p_0^{(1)}(\bv{x}_0)\prod_{k=1}^Np_k^{(1)}(\bf{x}_k|\bf{x}_{k-1})$ and $p^{(2)}(\bv{x}_0,\ldots,\bv{x}_N)  = p_0^{(2)}(\bv{x}_0)\prod_{k=1}^Np_k^{(2)}(\bf{x}_k|\bf{x}_{k-1})$. 
Then, the following result holds (see e.g. \cite{Gagliardi_D_et_Russo_G_IFAC2020_extended_Arxiv}):
\begin{Lemma}\label{lem:splitting_property}	
consider $p^{(1)}_{1:N}$ and $p^{(2)}_{1:N}$. Then: $
\mathcal{D}_{\KL}
\left( p^{(1)}_{1:N} || p^{(2)}_{1:N}  \right) =
\mathcal{D}_{\KL}
\left( p^{(1)}_{1:N-1}  || p^{(2)}_{1:N-1}  \right) +\mathbb{E}_{p^{(1)}_{1:N-1}}
\left[	
\mathcal{D}_{\KL} 
\left(p_N^{(1)}(\bv{x}_N|\bv{x}_{N-1}) || p_N^{(2)}(\bv{x}_N|\bv{x}_{N-1})
\right)
\right]$.
\end{Lemma}

\section{Formulation of the control problem}\label{sec:problem_formulation}
We now set-up and formalize the crowdsourcing problem.

\subsubsection{The agent and its target behavior}
we let $\bv{x}_k\in\mathcal{X}$ be the state of the agent at time $k$, $\bv{d}_{0:N}:=\{\bv{x}_0,\ldots,\bv{x}_N\}$ be a dataset illustrating a target behavior for the agent over the time horizon $\sT:={0:N}$ and $p(\bv{d}_{0:N}) = p(\bv{x}_0,\ldots,\bv{x}_N)$ be the joint pdf (or mdf) obtained from the observed data. We make the standard assumption that the Markov property holds. Then, following the chain rule for pdfs/mdfs, we have:
\begin{equation}\label{eqn:target}
\begin{split}
p_{0:N} & = p_0(\bv{x}_0)\prod_{k=1}^Np_k(\bv{x}_k|\bv{x}_{k-1}) = p_{0:0}\prod_{k=1}^N p_{k|k-1},
\end{split}
\end{equation}
which is simply termed as {\em target behavior} in what follows.
\begin{Remark}
 (\ref{eqn:target}) is a probabilistic description, a black box type model \cite{Gagliardi_D_et_Russo_G_IFAC2020_extended_Arxiv} that can always be obtained from the data.
\end{Remark}

We design the behavior of the agent by designing its joint pdf/mdf, i.e. $\pi(\bv{x}_0,\ldots,\bv{x}_N)$, which can be factorized as:
\begin{equation}\label{eqn:agent}
\begin{split}
\pi_{0:N}  & = \pi_{0:0}\prod_{k=1}^N \pi_{k|k-1},
\end{split}
\end{equation}
 $\pi_{k|k-1}:= \pi(\bv{x}_k|\bv{x}_{k-1})$. Hence, we can design the {\em agent behavior} $\pi(\bv{x}_0,\ldots,\bv{x}_N)$ by designing the sequence of $\pi_{k|k-1}$'s. 

\subsubsection{The agent-specific reward}
the goal of the agent is that of tracking the target behavior while, at the same time, optimizing an agent-specific reward function.  We let $r_k:\mathcal{X}\rightarrow\R$ be the reward obtained by an agent for being in state $\bv{x}_k$ at time $k$. Hence, the expected reward obtained by the agent that follows the behavior in (\ref{eqn:agent}) is given by:
\begin{equation}\label{eqn:reward}
\begin{split}
\E_{\pi_{k-1:k-1}}\left[\tilde{r}_k(\bv{X}_{k-1})\right] := \E_{\pi_{k-1:k-1}}\left[\E_{\pi_{k|k-1}}\left[r_k(\bv{X}_k)\right]\right].
\end{split}
\end{equation}

\subsubsection{The contributors}
we let $\sS:={1:S}$ be the set of contributors from which the agent seeks to crowdsource its behavior. In what follows, $\left\{\policy{k|k-1}{(i)}\right\}_{1:N}$ is the sequence of pdfs/mdfs (i.e. the behavior) provided by the $i$-th  contributor. The contributors do not necessarily know the specific target behavior and/or the specific reward of the agent that is crowdourcing. Also, the fact that the behavior of the contributors is specified as a pdf (or mdf) makes it possible to consider situations where the contributors might want to obfuscate their behavior with some noise to guarantee some level of e.g. differential privacy. We make the following assumption:
\begin{Assumption}\label{asn:source_policies}
The  $\policy{k|k-1}{(i)}$'s are such that $\DKL(\policy{k|k-1}{(i)}||p_{k|k-1})<+\infty$, $\forall k$ and $\forall i\in\sS$.
\end{Assumption}
\begin{Remark}
Assumption \ref{asn:source_policies} is not restrictive in practice. {As we shall see, the pdfs (mdfs) $\policy{k|k-1}{(i)}$'s are inputs received by the agent to craft its  behavior from the contributors (see e.g. Algorithm \ref{alg:problem_2}). Hence, upon receiving the $\policy{k|k-1}{(i)}$'s,  the agent can check Assumption \ref{asn:source_policies} for these pdfs (mdfs) and use only the $\policy{k|k-1}{(i)}$'s that satisfy such assumption.}
\end{Remark}

\subsubsection{Crowdsourcing as a control problem}\label{sec:crowd_noconstr}
 let $\alpha_k^{(i)}$, $i\in\sS$, be a weight and $\boldsymbol{\alpha}_k$ be the stack of the $\alpha_k^{(i)}$'s. Then, the crowdsourcing problem can be formalized via the following:
\begin{Problem}\label{prob:problem_2}
find the weights, ${\left\{\boldsymbol{\alpha}_k^\ast\right\}_{1:N}}$, such that:
\begin{equation}\label{eqn:problem_2}
    \begin{aligned}
&{\left\{\boldsymbol{\alpha}_k^\ast\right\}_{1:N}}\in \\
&    \underset{\left\{ \boldsymbol{\alpha}_k\right\}_{1:N}}{\text{arg min}}
    \DKL\left(\pi_{1:N}||p_{1:N}\right) - \sum_{k=1}^N\E_{\pi_{k-1:k-1}}\left[\tilde{r}_k(\bv{X}_{k-1})\right]\\
    s.t. & \pi_{k|k-1} = \sum_{i\in\sS}\alpha_k^{(i)}\policy{k|k-1}{(i)}, \ \ \ \forall k \\
    & \sum_{i\in\sS}\alpha_k^{(i)} = 1, \ \ \alpha_k^{(i)}\ge 0,  \ \forall i\in\sS, \ \ \forall k.
    \end{aligned}
\end{equation}
\end{Problem}
{
The following definition is used in the rest of the paper.
\begin{Definition}
we simply say that $\left\{\tilde{\pi}_{k|k-1}\right\}_{1:N}$, with $\tilde{\pi}_{k|k-1}:=\sum_{i\in\sS}{\alpha_k^{(i)}}^\ast\pi_{k|k-1}^{(i)}$ is: (i) an optimal solution of the crowdsourcing problem if the weights ${\alpha_k^{(i)}}^\ast$, $i\in\sS$, are a solution of (\ref{eqn:problem_2}); (ii) an approximate solution of the crowdsourcing problem if the weights are the solution of a problem that has the same constraints of (\ref{eqn:problem_2}) but a different cost function upper-bounding the cost function of (\ref{eqn:problem_2}).
\end{Definition}
\begin{Remark}
with Theorem \ref{thm:probl_2_sol} we show that an approximate solution  can be conveniently obtained by splitting Problem \ref{prob:problem_2} into sub-problems that are then solved recursively.
\end{Remark}
\begin{Remark}
Problem \ref{prob:problem_2} is a finite-horizon optimal control problem having $\boldsymbol{\alpha}_k$'s as decision variables. As we shall see, these are generated as feedback from the agent state and allow the construction of the pdfs/mdfs of the agent behavior.
\end{Remark}}
In Problem \ref{prob:problem_2}: (i) minimizing the first term in the cost  amounts at minimizing the discrepancy between the agent behavior and the target. Minimizing the second term maximizes the expected reward; (ii) the constraints formalize the fact that the agent behavior is crafted from the contributors (note how the constraints  guarantee  that $\pi_{k|k-1}\in\sD$, $\forall k$). 

\section{Tackling the crowdsourcing problem}\label{sec:tackling}
Inspired by \cite{Gagliardi_D_et_Russo_G_IFAC2020_extended_Arxiv}, we tackle Problem \ref{prob:problem_2} by splitting it in sub-problems that  can be solved recursively. Before introducing our main result we  give two technical lemmas
\begin{Lemma}\label{lem:non_convexity}
let $\gamma_k:\mathcal{X}\rightarrow \R$ be  integrable under the measure given by $\pi_{k|k-1}$ and consider
\begin{equation}\label{eqn:problem_non_convex}
    \begin{aligned}
    \underset{\left\{ {\beta}_{i}\right\}_{1:S}}{\text{min}}
    &\DKL\left(\pi_{k|k-1}||p_{k|k-1}\right) - \E_{\pi_{k|k-1}}\left[{\gamma}_k(\bv{X}_{k-1})\right]\\
    s.t. & \pi_{k|k-1} = \sum_{i\in\sS}\beta_{i}\policy{k|k-1}{(i)}, \\
    & \sum_{i\in\sS}\beta_{i} = 1, \ \ \ \beta_{i}\ge 0,  \ \ \forall i\in\sS.
    \end{aligned}
\end{equation}
The problem in (\ref{eqn:problem_non_convex}) is not convex.
\end{Lemma}
\begin{proof}
by using the first constraint in (\ref{eqn:problem_non_convex}) the cost function can be explicitly written as
$\int\sum_{i\in\sS}\beta^{(i)}\pi_{k|k-1}^{(i)}\left(\ln\left(\frac{\beta^{(i)}\pi_{k|k-1}^{(i)}}{p_{k|k-1}} - \gamma_k(\bv{x}_{k-1})\right)\right)d\bv{x}_k$.
The decision variables are the $\beta^{(i)}$'s and the cost is twice differentiable in these variables. A direct computation shows that the Hessian is, in general, an indefinite matrix.\end{proof}
\begin{Lemma}\label{lem:logsum}
let Assumption \ref{asn:source_policies} hold. Then:
$\DKL\left(\sum_{i\in\sS}\alpha_k^{(i)}\policy{k|k-1}{(i)}||p_{k|k-1}\right) \le \sum_{i\in\sS}\alpha_k^{(i)}\DKL\left(\policy{k|k-1}{(i)}||p_{k|k-1}\right)$.
\end{Lemma}
\begin{proof}
indeed 
$\DKL\left(\sum_{i\in\sS}\alpha_k^{(i)}\policy{k|k-1}{(i)}||p_{k|k-1}\right) \le \int\sum_{i\in\sS}\left(\alpha_k^{(i)}\policy{k|k-1}{(i)}\ln\frac{\alpha_k^{(i)}\policy{k|k-1}{(i)}}{p_{k|k-1}}\right)d\bv{x}_k \le \sum_{i\in\sS}\alpha_k^{(i)}\int \policy{k|k-1}{(i)}\ln\frac{\policy{k|k-1}{(i)}}{p_{k|k-1}}d\bv{x}_k$, which is in fact $\sum_{i\in\sS}\alpha_k^{(i)}\DKL\left(\policy{k|k-1}{(i)}||p_{k|k-1}\right)$. In the above expression, the first inequality follows from the log-sum inequality. Instead, the last inequality is obtained by using monotonicity of the logarithm, Assumption \ref{asn:source_policies}, the fact that $\alpha_k^{(i)}$ is independent on $\bv{x}_k$ and that $0\le\alpha_k^{(i)}\le 1$, $\forall k$ and $\forall i$.\end{proof}
\subsection{Synthesizing behaviors from the contributors}
We now prove the following result to crowdsource behaviors from  contributors (see Remark \ref{rem:pmfs} for discrete variables).
\begin{Theorem}\label{thm:probl_2_sol}
consider Problem \ref{prob:problem_2}. Then:
\begin{enumerate}
\item {$\left\{\tilde{\pi}_{k|k-1}\right\}_{1:N}$, with} $\tilde{\pi}_{k|k-1} =\sum_{i\in\sS}{\alpha_k^{(i)}}^\ast\policy{k|k-1}{(i)}$ and 
\begin{equation}\label{eqn:sub_opt_probl_2}
    \begin{aligned}
    \boldsymbol{\alpha}_k^\ast \in \underset{\boldsymbol{\alpha}_k}{\text{arg min}}
    &\ \ \bv{a}_k^T(\bv{x}_{k-1})\boldsymbol{\alpha}_k\\
    s.t. & \sum_{i\in\sS}\alpha_k^{(i)} = 1, \ \alpha_k^{(i)}\ge 0,  \ \ \forall i\in\sS,
    \end{aligned}
\end{equation}
is an {approximate} solution of the crowdsourcing problem. In (\ref{eqn:sub_opt_probl_2}):
\begin{equation}\label{eqn:vector_coeffs}
\begin{split}
 \bv{a}_k(\bv{x}_{k-1}) &:= [a_k^{(1)}(\bv{x}_{k-1}),\ldots,a_k^{(S)}(\bv{x}_{k-1})]^T, \\
 \bv{a}_k^{(i)}(\bv{x}_{k-1}) & := \DKL\left(\policy{k|k-1}{(i)}||p_{k|k-1}^{(i)}\right) \\
 & - \E_{\pi^{(i)}_{k|k-1}}\left[\bar{r}_k({\bv{X}_{k}})\right],
\end{split}
\end{equation}
and $\bar{r}_k({\bv{X}_{k}})$ is obtained via backward recursion as
\begin{equation}\label{eqn:policy_backward_probl_2}
\begin{split}
& \bar{r}_k(\bv{x}_k) := r_k(\bv{x}_k) + \hat{r}_k(\bv{x}_k), \\ 
& \hat{r}_k(\bv{x}_k) = - \bv{a}^T_{k+1}(\bv{x}_k)\boldsymbol{\alpha}^\ast_{k+1}, \ \ \hat{r}_N(\bv{x}_N) = 0;
\end{split}
\end{equation}
\item the cost corresponding to $\tilde{\pi}_{k|k-1}$ is upper bounded by
\begin{equation}\label{eqn:optimal_cost_probl_2}
\sum_{k=1}^N\E_{{\tilde{\pi}}_{k-1:k-1}}\left[\bv{a}_k^T(\bv{X}_{k-1})\boldsymbol{\alpha}_k^\ast\right].
\end{equation}
\end{enumerate}
\end{Theorem}
\begin{proof}
the result is proven via induction. First ({\bf Step 1}) we show how the problem can be decomposed into sub-problems, where the last one ($k=N$) can be tackled independently from the others. In {\bf Step 2} we get a convex upper bound of the cost function that leads to a linear problem. We then take the minimum from this problem and show that ({\bf Step 3})  we can again split the problem. In particular, we get again a convex upper bound of the cost function at  $k=N-1$ and show that the resulting problem can be solved independently on the others. The structure of the problem at $k=N-1$ is the same as the problem at $k=N$. By tackling the problem at $k=N-1$ we show, in {\bf Step 4}, how $\forall k= 1,\ldots, N-2$, $\tilde\pi_{k|k-1}$ can be obtained from a linear problem that always has the same structure. This leads to the desired conclusion. 

\noindent {\bf Step 1.} The cost in (\ref{eqn:problem_2}) can be written via Lemma \ref{lem:splitting_property} as:
\begin{equation}\label{eqn:split_cost_problem_1}
\begin{split}
& \DKL\left(\pi_{1:N-1}||p_{1:N-1}\right)- \sum_{k=1}^{N-1}\E_{\pi_{k-1:k-1}}\left[\tilde{r}_k(\bv{X}_{k-1})\right]\\
& + \E_{\pi_{1:N-1}}\left[\DKL\left(\pi_{N|N-1}||p_{N|N-1}\right)\right] \\
& - \E_{\pi_{N-1:N-1}}\left[\tilde{r}_N(\bv{X}_{N-1})\right].
\end{split}
\end{equation}
Moreover, note that $\E_{\pi_{1:N-1}}\left[\DKL\left(\pi_{N|N-1}||p_{N|N-1}\right)\right] = \E_{\pi_{N-1:N-1}}\left[\DKL\left(\pi_{N|N-1}||p_{N|N-1}\right)\right]$ and therefore the last two terms in (\ref{eqn:split_cost_problem_1}) become $\E_{\pi_{N-1:N-1}}\left[\DKL\left(\pi_{N|N-1}||p_{N|N-1}\right) - \tilde{r}_N(\bv{X}_{N-1}) \right]$.
Hence, Problem \ref{prob:problem_2} can be recast as the sum of two sub-problems:
\begin{subequations}
\begin{equation}\label{eqn:split_1_problem_2}
    \begin{aligned}
    \underset{\left\{ \boldsymbol{\alpha}_k\right\}_{1:N-1}}{\text{min}}
    &\DKL\left(\pi_{1:N-1}||p_{1:N-1}\right)- \sum_{k=1}^{N-1}\E_{\pi_{k-1:k-1}}\left[\tilde{r}_k(\bv{X}_{k-1})\right]\\
    s.t. & \  \pi_{k|k-1} = \sum_{i\in\sS}\alpha_k^{(i)}\policy{k|k-1}{(i)}, \ \ \ \forall k\in1:N-1, \\
    & \sum_{i\in\sS}\alpha_k^{(i)} = 1, \  \alpha_k^{(i)}\ge 0, \ \forall i\in\sS, \ \forall k\in 1:N-1,\\
    \end{aligned} 
\end{equation}
\text{and}
\begin{equation}\label{eqn:split_2_problem_2}
    \begin{aligned}
    \underset{\boldsymbol{\alpha}_N}{\text{min}}
    &\ \E_{\pi_{N-1:N-1}}\left[\DKL\left(\pi_{N|N-1}||p_{N|N-1}\right) - \tilde{r}_N(\bv{X}_{N-1}) \right]\\
    s.t. & \  \pi_{N|N-1} = \sum_{i\in\sS}\alpha_k^{(i)}\policy{N|N-1}{(i)},\\
    & \sum_{i\in\sS}\alpha_N^{(i)} = 1, \ \ \alpha_N^{(i)}\ge 0,  \ \ \forall i\in\sS.
    \end{aligned} 
\end{equation}
\end{subequations} 

That is, Problem \ref{prob:problem_2} can be approached by first solving (\ref{eqn:split_2_problem_2}) and then by taking into account the solution to solve (\ref{eqn:split_1_problem_2}).

\noindent {\bf Step 2.} Let $c_N(\bv{x}_{N-1}):=\DKL\left(\pi_{N|N-1}||p_{N|N-1}\right) - \tilde{r}_N(\bv{x}_{N-1})$. Then, by using linearity of the expectation and the fact that the decision variable is independent on the pdf over which the expectation is taken, we have that solving (\ref{eqn:split_2_problem_2}) is equivalent to find $\E_{\pi_{N-1:N-1}}\left[c_N^\ast(\bf{X}_{N-1}))\right]$, where  $c_N^\ast(\bf{x}_{N-1})$ is the optimal cost obtained by solving:
\begin{equation}\label{eqn:step_intermediate_problem_2}
    \begin{aligned}
    \underset{\boldsymbol{\alpha}_N}{\text{min}}
    & \ c_N(\bv{x}_{N-1})\\
    s.t. & \  \pi_{N|N-1} = \sum_{i\in\sS}\alpha_{N}^{(i)}\policy{N|N-1}{(i)},\\
    & \sum_{i\in\sS}\alpha_N^{(i)} = 1, \ \ \alpha_N^{(i)}\ge 0,  \ \ \forall i\in\sS.
    \end{aligned} 
\end{equation}
Then, by embedding the first constraint in (\ref{eqn:step_intermediate_problem_2}) in its cost function we have the following equivalent formulation
\begin{equation}\label{eqn:step_2_problem_2}
    \begin{aligned}
    \underset{\boldsymbol{\alpha}_N}{\text{min}}
    & \  \DKL\left(\sum_{i\in\sS}\alpha_N^{(i)}\policy{N|N-1}{(i)}||p_{N|N-1}\right) - \tilde{r}_N(\bv{x}_{N-1}) \\
    s.t. & \   \sum_{i\in\sS}\alpha_N^{(i)} = 1, \ \ \alpha_N^{(i)}\ge 0,  \ \ \forall i\in\sS,
    \end{aligned} 
\end{equation}
where we used the definition of $c_N(\bv{x}_{N-1})$. Lemma \ref{lem:non_convexity} implies the non-convexity of (\ref{eqn:step_2_problem_2}) and  we now  find a convex approximation upper bounding its cost function. Specifically:
\begin{equation*}
\begin{split}
& \DKL\left(\sum_{i\in\sS}\alpha_N^{(i)}\policy{N|N-1}{(i)}||p_{N|N-1}\right)- \tilde{r}_N(\bv{x}_{N-1})\\
& \le \sum_{i\in\sS}{\alpha}_N^{(i)}\DKL\left(\policy{N|N-1}{(i)}||p_{N|N-1}\right)- \tilde{r}_N(\bv{x}_{N-1})\\
& = \sum_{i\in\sS}\left(\alpha_N^{(i)}\int \policy{N|N-1}{(i)}\ln\frac{\policy{N|N-1}{(i)}}{p_{N|N-1}^{(i)}}d\bv{x}_N\right)\\
& - \int\sum_{i\in\sS}\alpha_N^{(i)}\policy{N|N-1}{(i)}\bar{r}_n(\bv{x}_N)d\bv{x}_N\\
&= \sum_{i\in\sS}a^{(i)}_N(\bv{x}_{N-1})\alpha_N^{(i)} = \bv{a}_N^T(\bv{x}_{N-1})\boldsymbol{\alpha}_N:=\bar{c}_N(\bv{x}_{N-1}),
\end{split}
\end{equation*}
where we used Lemma \ref{lem:logsum}, the definition of $a^{(i)}_N(\bv{x}_{N-1})$ given in (\ref{eqn:vector_coeffs}) and the fact that the set $\sS$ has a finite cardinality. Also, in the above expression: (i) $\bar{r}_N(\bv{x}_N) := r_N(\bv{x}_N) + \hat{r}_N(\bv{x}_N)$, with $\hat{r}_N(\bv{x}_N) = 0$. This corresponds to the initialization of the backward recursion in (\ref{eqn:policy_backward_probl_2}); (ii) the upper bound $\sum_{i\in\sS}a^{(i)}_N(\bv{x}_{N-1})\alpha_N^{(i)}$ is the cost function for the problem in (\ref{eqn:sub_opt_probl_2}) at $k=N$. Hence, by solving the linear problem in (\ref{eqn:sub_opt_probl_2}) at $k=N$ we obtain the vector $ \boldsymbol{\alpha}_N^\ast$, which yields the pdf $\tilde{\pi}_{N|N-1} = \sum_{i\in\sS}{\alpha_N^{(i)}}^\ast\policy{N|N-1}{(i)}$ given in part 1) of the statement for $k=N$. Moreover, the corresponding minimum for (\ref{eqn:sub_opt_probl_2}) at $k=N$ is $\bv{a}_N^T(\bv{x_{N-1}})\boldsymbol{\alpha}_N^\ast$. Hence, the optimal cost of (\ref{eqn:split_2_problem_2}) is upper bounded by $\E_{\policy{N-1:N-1}{}}\left[\bv{a}_N^T(\bv{X_{N-1}})\boldsymbol{\alpha}_N^\ast\right]$.

\noindent {\bf Step 3.} Since the original problem has been reformulated as the sum of the two sub-problems (\ref{eqn:split_1_problem_2}) and (\ref{eqn:split_2_problem_2}) we have that the cost function of Problem \ref{prob:problem_2} is upper bounded by
\begin{equation*}
\begin{split}
&\DKL\left(\pi_{1:N-1}||p_{1:N-1}\right)- \sum_{k=1}^{N-1}\E_{\pi_{k-1:k-1}}\left[\tilde{r}_k(\bv{X}_{k-1})\right] \\
&+ \E_{\policy{N-1:N-1}{}}\left[\bv{a}_N^T(\bv{X_{N-1}})\boldsymbol{\alpha}_N^\ast\right],
\end{split}
\end{equation*}
which, by means of Lemma \ref{lem:logsum} can be expressed as:
\begin{equation}\label{eqn:rearrange_cost_prob_2}
\begin{split}
&\DKL\left(\pi_{1:N-2}||p_{1:N-2}\right)- \sum_{k=1}^{N-2}\E_{\pi_{k-1:k-1}}\left[\tilde{r}_k(\bv{X}_{k-1})\right] \\
& + \E_{\pi_{1:N-2}}\DKL(\pi_{N-1|N-2}||p_{N-1|N-2})\\
& - \E_{\pi_{N-2:N-2}}\left[\tilde{r}_{N-1}(\bv{X}_{N-2})\right]\\
&+ \E_{\policy{N-1:N-1}{}}\left[\bv{a}_N^T(\bv{X_{N-1}})\boldsymbol{\alpha}_N^\ast\right].
\end{split}
\end{equation}
Now, for the above expression we note the following. First:
\begin{equation*}
\begin{split}
&\E_{\pi_{1:N-2}}\DKL(\pi_{N-1|N-2}||p_{N-1|N-2})\\
&= \E_{\pi_{N-2:N-2}}\DKL(\pi_{N-1|N-2}||p_{N-1|N-2}).
\end{split}
\end{equation*}
Also:
\begin{equation*}
\begin{split}
&\E_{\policy{N-1:N-1}{}}\left[\bv{a}_N^T(\bv{X_{N-1}})\boldsymbol{\alpha}_N^\ast\right]\\
& = \E_{\pi_{N-2:N-2}}\left[\E_{\pi_{N-1|N-2}}\left[\bv{a}_N^T(\bv{X_{N-1}})\boldsymbol{\alpha}_N^\ast\right]\right].
\end{split}
\end{equation*}
The above equality was obtained by using the definition of expectation after noting that: (i) $\bv{a}_N^T(\bv{X_{N-1}})\boldsymbol{\alpha}_N^\ast$ only depends on $\bv{X_{N-1}}$ and hence $\E_{\policy{N-1:N-1}{}}\left[\bv{a}_N^T(\bv{X_{N-1}})\boldsymbol{\alpha}_N^\ast\right]=\E_{\policy{N-2:N-1}{}}\left[\bv{a}_N^T(\bv{X_{N-1}})\boldsymbol{\alpha}_N^\ast\right]$; (ii) $\pi_{N-2:N-1} = \pi_{N-1|N-2}\pi_{N-2:N-2}$.
Thus, the last three terms in (\ref{eqn:rearrange_cost_prob_2}) become $\E_{\pi_{N-2:N-2}}\left[\DKL(\pi_{N-1|N-2}||p_{N-1|N-2})\right]-\E_{\pi_{N-2:N-2}}\left[\E_{\pi_{N-1|N-2}}\left[r_{N-1}(\bv{X}_{N-1})-\bv{a}_N^T(\bv{X_{N-1}})\boldsymbol{\alpha}_N^\ast\right]\right]$.
As a result, (\ref{eqn:step_intermediate_problem_2}) can be  formulated as the sum of:
\begin{subequations}
\begin{equation}\label{eqn:split_1_problem3_probl2}
    \begin{aligned}
    \underset{\left\{ \boldsymbol{\alpha}_k\right\}_{1:N-2}}{\text{min}}
    &\DKL\left(\pi_{1:N-2}||p_{1:N-2}\right)- \sum_{k=1}^{N-2}\E_{\pi_{k-1:k-1}}\left[\tilde{r}_k(\bv{X}_{k-1})\right]\\
    s.t. & \ \pi_{k|k-1} = \sum_{i\in\sS}\alpha_k^{(i)}\policy{k|k-1}{(i)}, \ \ \ \forall  k\in 1:N-2,\\
    & \sum_{i\in\sS}\alpha_k^{(i)} = 1, \ \ \alpha_k^{(i)}\ge 0, \ \forall i\in\sS, \ \forall k\in1:N-2,\\
    \end{aligned} 
\end{equation}
\text{and} 
\begin{equation}\label{eqn:split_2_problem_4_probl2}
    \begin{aligned}
    \underset{\boldsymbol{\alpha}_{N-1}}{\text{min}}
    &\E_{\pi_{N-2:N-2}}\left[c_{N-1}(\bv{X}_{N-2}) \right]\\
    s.t. & \  \pi_{N-1|N-2} = \sum_{i\in\sS}\alpha_{N-1}^{(i)}\policy{N-1|N-2}{(i)},\\
    & \sum_{i\in\sS}\alpha_{N-1}^{(i)} = 1, \ \ \alpha_{N-1}^{(i)}\ge 0,  \ \ \forall i\in\sS,
    \end{aligned} 
\end{equation}
\end{subequations}
with $c_{N-1}(\bv{x}_{N-2}):=\DKL\left(\pi_{N-1|N-2}||p_{N-1|N-2}\right) - \E_{\pi_{N-1|N-2}}\left[\bar{r}_{N-1}(\bv{x}_{N-1})\right]$, where 
\begin{equation}\label{eqn:iteration_N-1_probl2}
\begin{split}
\bar{r}_{N-1}(\bv{X}_{N-1}) &= {r}_{N-1}(\bv{X}_{N-1}) + \hat{r}_{N-1}(\bv{X}_{N-1}),\\
\hat{r}_{N-1}(\bv{x}_{N-1} & ):=-\bv{a}_N^T(\bv{X_{N-1}})\boldsymbol{\alpha}_N^\ast.
\end{split}
\end{equation}
Note that (\ref{eqn:iteration_N-1_probl2}) corresponds to (\ref{eqn:policy_backward_probl_2}) at $k=N-1$. Also, by following  the same arguments used in Step $2$, we have that (\ref{eqn:split_2_problem_4_probl2}) is equivalent to find $\E_{\pi_{N-2:N-2}}\left[c_{N-1}^\ast(\bf{X}_{N-2}))\right]$, where $c_{N-1}^\ast(\bf{X}_{N-2})$ is the optimal cost of
\begin{equation}\label{eqn:step_intermediate_problem_2_bis}
    \begin{aligned}
    \underset{\boldsymbol{\alpha}_N}{\text{min}}
    & \ c_{N-1}(\bv{x}_{N-2})\\
    s.t. & \  \pi_{N-1|N-2} = \sum_{i\in\sS}\alpha_{N-1}^{(i)}\policy{N-1|N-2}{(i)},\\
    & \sum_{i\in\sS}\alpha_{N-1}^{(i)} = 1, \ \ \alpha_{N-1}^{(i)}\ge 0,  \ \ \forall i\in\sS.
    \end{aligned} 
\end{equation}
Again, Lemma \ref{lem:non_convexity} implies that (\ref{eqn:step_intermediate_problem_2_bis}) is not convex and its cost function can be upper bounded via Lemma \ref{lem:logsum}, yielding
\begin{equation*}
    \begin{aligned}
    \underset{\boldsymbol{\alpha}_{N-1}}{\text{min}}
    &\ \ \bv{a}_{N-1}^T(\bv{x}_{N-2})\boldsymbol{\alpha}_{N-1}\\
    s.t. & \sum_{i\in\sS}\alpha_{N-1}^{(i)} = 1, \ \ \alpha_{N-1}^{(i)}\ge 0,  \ \ \forall i\in\sS,
    \end{aligned}
\end{equation*}
where $\boldsymbol{\alpha}_{N-1}(\bv{x}_{N-2}) = [\alpha_{N-1}^{(1)}(\bv{x}_{N-2}),\ldots, \alpha_{N-1}^{(S)}(\bv{x}_{N-2}) ]^T$ is defined as in (\ref{eqn:vector_coeffs}). The above problem again corresponds to the linear problem in (\ref{eqn:sub_opt_probl_2}) at $k=N-1$.  Solving the above problem leads to $ \boldsymbol{\alpha}_{N-1}^\ast$ and hence to $\tilde{\pi}_{N-1|N-2} = \sum_{i\in\sS}{\alpha_{N-1}^{(i)}}^\ast\policy{N-1|N-2}{(i)}$, which is  given in part i) of the statement for $k=N-1$. Moreover, the optimal cost of  (\ref{eqn:split_2_problem_4_probl2}) is upper bounded by $\E_{\policy{N-2:N-2}{}}\left[\bv{a}_{N-1}^T(\bv{X_{N-2}})\boldsymbol{\alpha}_{N-1}^\ast\right]$.
 
 \noindent{\bf Step 4.} By iterating Step $3$ we find that, at each of the remaining time steps, i.e. $\forall k \in 1:N-2$, Problem \ref{prob:problem_2}  can always be split in sub-problems, where the sub-problem corresponding to the last time instant in the window can be solved independently on the others and takes the form:
 \begin{equation}\label{eqn:cost_final}
    \begin{aligned}
    \underset{\boldsymbol{\alpha}_k}{\text{min}}
    &\ \E_{\pi_{k-1:k-1}}\left[c_k(\bv{X}_{k-1}) \right]\\
    s.t. & \  \pi_{k|k-1} = \sum_{i\in\sS}\alpha_k^{(i)}\policy{k|k-1}{(i)},\\
    & \sum_{i\in\sS}\alpha_k^{(i)} = 1, \ \ \alpha_k^{(i)}\ge 0,  \ \ \forall i\in\sS,
    \end{aligned}    
\end{equation}
with $c_k(\bv{X}_{k-1}): =\DKL\left(\pi_{k|k-1}||p_{k|k-1}\right) - \E_{\pi_{k|k-1}}\left[\bar{r}_k(\bv{x}_{k})\right]$. Again, its optimal cost is upper bounded by $\E_{\pi_{k-1:k-1}}\left[\bar{c}_k^\ast(\bf{X}_{k-1}))\right]$, where $\bar{c}_k^\ast(\bf{x}_{k-1})$  is the optimal cost of
 \begin{equation}\label{eqn:split_last_problem_2}
    \begin{aligned}
     \underset{\boldsymbol{\alpha}_k}{\text{min}}
    &\ \ \bv{a}_k^T(\bv{x}_{k-1})\boldsymbol{\alpha}_k\\
    s.t. & \sum_{i\in\sS}\alpha_k^{(i)} = 1, \ \ \alpha_k^{(i)}\ge 0,  \ \ \forall i\in\sS,
    \end{aligned}
\end{equation}
with $\bv{a}_k(\bv{x}_{k-1}):= [a_k^{(1)}(\bv{x}_{k-1}),\ldots,a_k^{(S)}(\bv{x}_{k-1})]^T$, 
\begin{equation}\label{eqn:policy_generictime_problem_2}
\bv{a}_k^{(i)}(\bv{x}_{k-1}) := \DKL\left(\policy{k|k-1}{(i)}||p_{k|k-1}^{(i)}\right) - \E_{\pi^{(i)}_{k|k-1}}\left[\bar{r}_k({\bv{X}_{k}})\right],
\end{equation}
and $\bar{r}_{k}(\bv{X}_{k}) = {r}(\bv{X}_{k}) + \hat{r}(\bv{X}_{k})$, $\hat{r}_k(\bv{x}_k) = - \bv{a}^T_{k+1}(\bv{x}_k)\boldsymbol{\alpha}^\ast_{k+1}$. This is (\ref{eqn:sub_opt_probl_2}) which yields, at  $k$,  $\tilde{\pi}_{k|k-1} =\sum_{i\in\sS}{\alpha_k^{(i)}}^\ast\policy{k|k-1}{(i)}$, given in part 1) of the statement. Moreover, the optimal cost of (\ref{eqn:cost_final}) is upper bounded by $
\E_{\policy{k-1:k-1}{}}\left[\bv{a}_k^T(\bv{X}_{k-1})\boldsymbol{\alpha}_k^\ast\right]$. 
The desired conclusions then follow.
\end{proof}

We now also give the following:
\begin{Corollary}\label{cor:probl_2_behavior}
consider the sequence of pdfs $\left\{\tilde\pi_{k|k-1}\right\}_{1:N}$  from Theorem \ref{thm:probl_2_sol}. Then: (i) at each $k$, $\tilde\pi_{k|k-1} = \pi_{k|k-1}^{(j_k)}$, where $j_k\in\sS$ is the index corresponding to the smallest element of the vector defined in (\ref{eqn:vector_coeffs}); (ii) the cost obtained by following $\left\{\tilde\pi_{k|k-1}\right\}_{1:N}$ is given by $\DKL\left(\tilde\pi_{1:N}||p_{1:N}\right) - \sum_{k=1}^N\E_{\tilde\pi_{k-1:k-1}}\left[\tilde{r}_k(\bv{X}_{k-1})\right]$, where $\tilde\pi_{1:N}:= \prod_{k=1}^N\pi_{k|k-1}^{(j_k)}$.
\end{Corollary}
\begin{proof}
the first part of the statement follows from the fact that the coefficients $\boldsymbol{\alpha}_k^\ast$ are obtained by solving, at each $k$, the linear problem (\ref{eqn:sub_opt_probl_2}): this has the standard simplex as feasibility domain. The second part directly follows from evaluating the cost of Problem \ref{prob:problem_2} when $\tilde\pi_{k|k-1} = \pi_{k|k-1}^{(j_k)}$.
\end{proof}
\begin{Remark}
Corollary \ref{cor:probl_2_behavior} implies that, by using Theorem \ref{thm:probl_2_sol}, an agent can obtain better results than the results that would be obtained by the individual contributors. Moreover, this can be achieved by properly selecting one contributor at each $k$.
\end{Remark}
\begin{Remark}\label{rem:pmfs}
Theorem \ref{thm:probl_2_sol} is stated for pdfs. The same result holds for pmfs. In particular, the same proof can be used with the only difference being the definition of the \KL-divergence. Since the formulation of the result and the main steps  of its proof are the same, the statement for pmfs is omitted here.
\end{Remark}
{\begin{Remark}
Theorem \ref{thm:probl_2_sol} relies on upper bounding,  $\forall k$,  $\DKL(\sum_{i\in\sS}\alpha_k^{(i)}\policy{k|k-1}{(i)}||p_{k|k-1})$ via Lemma \ref{lem:logsum}. Finding a solution that does not rely on any approximation requires computing this quantity. In general, this can become analytically intractable and/or computationally inefficient \cite{4218101}.
\end{Remark}}

\subsection{An algorithm from Theorem \ref{thm:probl_2_sol} and Corollary \ref{cor:probl_2_behavior}}\label{sec:algorithms}
Theorem \ref{thm:probl_2_sol} and Corollary \ref{cor:probl_2_behavior}  can be turned into an algorithmic procedure with its steps given in Algorithm \ref{alg:problem_2}.

\begin{algorithm}[H]
	\caption{Pseudo-code from Theorem \ref{thm:probl_2_sol} and Corollary \ref{cor:probl_2_behavior}} \label{alg:problem_2}
	\begin{algorithmic}	[1]
			\State \textbf{Inputs:}  time-horizon $\sT$, target behavior $p_{0:N}$,  reward $r_k(\cdot)$, contributors behaviors $\policy{k|k-1}{(i)}$
		\State \textbf{Output:} $\left\{\tilde{\pi}_{k|k-1}\right\}_{1:N}$ for Problem \ref{prob:problem_2}
		\State \textbf{Initialize:}  $\bv{a}^T_{N+1}(\bv{x}_N) \gets 0$ and $\boldsymbol{\alpha}^\ast_{N+1} \gets 0$;
		\State \textbf{Main loop:} 
		\For{$ k = N$  to $1$}		 
		\State $\hat{r}_k(\bv{x}_k) \gets - \bv{a}^T_{k+1}(\bv{x}_k)\boldsymbol{\alpha}^\ast_{k+1}$
		\State $\bar{r}_k(\bv{x}_k) \gets r_k(\bv{x}_k) + \hat{r}_k(\bv{x}_k)$
		\For{$i=1$ to $S$}\\
		 $ {a}_k^{(i)}(\bv{x}_{k-1}) {\tiny \gets} \DKL\left(\policy{k|k-1}{(i)}||p_{k|k-1}\right) { -} \E_{{\pi_{k|k-1}^{(i)}}}\left[\bar{r}_k({\bv{X}_{k}})\right]$
		\EndFor
		\State $j \gets$ index corresponding to smallest ${a}_k^{(i)}(\bv{x}_{k-1})$
				\State {$\boldsymbol{\alpha}^\ast_{k}\gets[\alpha_k^{(1)},\ldots,\alpha_k^{(S)}]^T$, $\alpha_k^{(i)} = 0$, $\forall i\ne j$, $\alpha_k^{(j)} = 1$}
		\State $\tilde{\pi}_{k|k-1} \gets \policy{k|k-1}{(j)}$
	\EndFor
	\end{algorithmic}
\end{algorithm}

\begin{Remark}
intuitively, Algorithm \ref{alg:problem_2} implements a {\em context-aware switch}. The switch, at each $k$, selects one contributor  by just picking the smallest element in the vector $\bv{a}_k(\bv{x}_{k-1})$.
\end{Remark}

\section{Numerical Example}\label{sec:example}
We illustrate our results via a simple example {inspired by \cite{7795769}, where a parsing engine for connected vehicles was developed to orchestrate the routes of a network of cars (routes were determined by a shared cloud service from the available information). In this context, we now consider the scenario where an agent (i.e. a connected car) travels in a given geographic area and seeks to crowdsource from other cars in the same area (i.e. the contributors) a route that would minimize a given cost index related to both traffic conditions and passengers' preferences. In the simplified scenario considered here, the agent}  travels along the graph of Fig. \ref{fig:graph_and_reward} {(links might be roads connecting e.g. intersections)}. We implemented Algorithm \ref{alg:problem_2} in Matlab to crowdsource a route enabling the agent to go from node $1$ to $6$.  Specifically, $N=4$,  $x_k\in\mathcal{X}:={1:6}$ is the position of the agent on the graph at time $k$ and $r(\cdot)$ is the reward associated to a given position/node {(e.g. related to the traffic to reach a  node)}. The pmfs $p_{k|k-1}$'s for the target behavior  {(e.g. the {preferred} route of the passengers)} underpin a route that leads to $6$ via $2$, $4$, $5$. In the example, there are $2$ contributors: the routes highlighed in Fig. \ref{fig:graph_and_reward} are a realization drawn, at each $k$, from the pmfs $\pi^{(i)}_{k|k-1}$'s, $i\in\{1, 2\}$\footnote{All pmfs are given at \url{https://github.com/GIOVRUSSO/Control-Group-Code/tree/master/Crowdsourcing}.}. 
\begin{figure}[thbp]
	\centering	
	\includegraphics[width=0.5\linewidth]{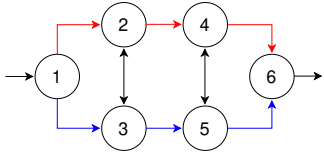}
	\caption{Graph over which the agent is traveling. The route drawn from $\pi^{(1)}_{k|k-1}$'s is  in red, the route from $\pi^{(2)}_{k|k-1}$'s is in blue (colors online). }
	\label{fig:graph_and_reward}
\end{figure}
First, we considered the agent-specific reward in the top-left panel of Fig. \ref{fig:sim_run_1}, which favors node $2$ over node $3$. The algorithm returned the pmfs $\tilde\pi_{k|k-1}$'s and the route of the agent was obtained by sampling from these pmfs at each $k$ (the route is  in the bottom panel of Fig. \ref{fig:sim_run_1} - see the top-right panel of the figure for the corresponding pmfs). We then considered the reward of Fig. \ref{fig:sim_run_2}, which favors node $3$ over node $2$. The algorithm returned a sequence of pmfs underlying the route that this time passes through node $3$ (see  Fig. \ref{fig:sim_run_2}).
\begin{figure}[thbp]
	\centering	
	\psfrag{x1}[c]{{Nodes}}
	\psfrag{y1}[c]{{Reward}}
	\psfrag{x2}[c]{{$k$}}
	\psfrag{y2}[c]{{$\mathcal{X}$}}
	\psfrag{z2}[c]{{$\tilde\pi_{k|k-1}$}}
	\psfrag{x3}[c]{{$k$}}
	\psfrag{y3}[c]{{$x_k$}}	
	\includegraphics[width=0.73\linewidth]{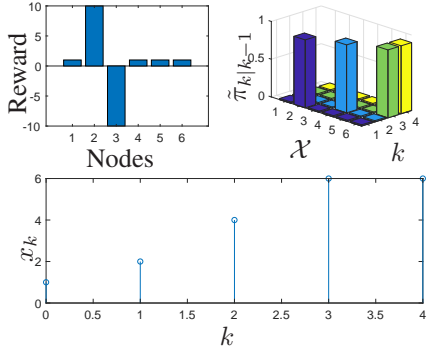}
	\caption{Top-left: agent-specific reward. The agent route obtained by drawing from the pmfs $\tilde\pi_{k|k-1}$'s (top-right) is shown in the bottom panel (the agent moves through node $1$, $2$, $4$ and reaches $6$ at $k =3$).}
	\label{fig:sim_run_1}
\end{figure}

\begin{figure}[thbp]
	\centering	
	\psfrag{x1}[c]{{Nodes}}
	\psfrag{y1}[c]{{Reward}}
	\psfrag{x2}[c]{{$k$}}
	\psfrag{y2}[c]{{$\mathcal{X}$}}
	\psfrag{z2}[c]{{$\tilde\pi_{k|k-1}$}}
	\psfrag{x3}[c]{{$k$}}
	\psfrag{y3}[c]{{$x_k$}}	
	\includegraphics[width=0.73\linewidth]{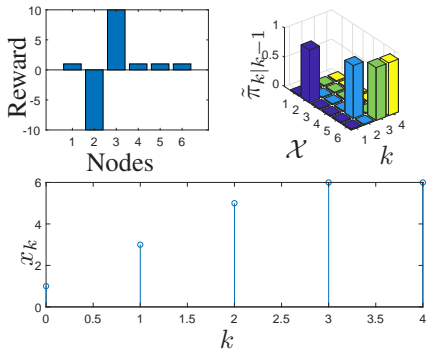}
	\caption{Top-left: agent-specific reward. The agent route obtained by drawing from the pmfs $\tilde\pi_{k|k-1}$'s (top-right) is shown in the bottom panel (now the agent moves through node $1$, $3$, $5$ and  gets to $6$ at $k =3$).}
	\label{fig:sim_run_2}
\end{figure}

\section{Conclusions}
We considered the problem of designing, from data, agents that are able to craft their behavior from a number of contributors in order to fulfill some task. After formalizing the crowdsourcing process as a control problem, we presented a result to compute the agent behavior from the information made available by the contributors. The result was  turned into an algorithmic procedure and our findings were complemented via an example.  {Current research includes: (i) finding special cases for which a solution to Problem \ref{prob:problem_2} can be found without any approximation; (ii) designing a scalable sharing platform to support Algorithm \ref{alg:problem_2}.  Within this platform, the behavior will be computed on the agent side. We are building,} from the example of Section \ref{sec:example}, a hardware-in-the-loop (HiL) demo. {In the demo, Algorithm \ref{alg:problem_2} will be deployed on a real car, which will crowdsource a route from a large number of {\em contributing} cars generated by the HiL infrastructure of \cite{griggs2019vehicle}}. 
\subsubsection*{Acknowledgments}
I would like to thank Dr. Gagliardi and Prof. Naoum-Sawaya at Ivey Business School for reading an earlier version of the results. {I am also grateful to the AE and the anonymous reviewers for their constructive feedback.}

\end{document}